\documentclass[12 pt]{article}%
\usepackage{amsmath, amsfonts, amsthm, color,latexsym}
\usepackage{amsmath}
\usepackage{amsfonts}
\usepackage{amssymb}
\usepackage{color, soul}
\usepackage[all]{xy}
\usepackage{graphicx}%
\providecommand{\U}[1]{\protect\rule{.1in}{.1in}}
\allowdisplaybreaks[4]
\newtheorem{theorem}{Theorem}[section]
\newtheorem{proposition}[theorem]{Proposition}
\newtheorem{corollary}[theorem]{Corollary}
\newtheorem{example}[theorem]{Example}
\newtheorem{examples}[theorem]{Examples}
\newtheorem{remark}[theorem]{Remark}

\newtheorem{lemma}[theorem]{Lemma}
\newtheorem{final remark}[theorem]{Final Remark}
\newtheorem{definition}[theorem]{Definition}
\allowdisplaybreaks[4]

\begin{document}

\title{Aron-Berner extensions of almost Dunford-Pettis multilinear operators}
\author{Geraldo Botelho\thanks{Supported by CNPq Grant
304262/2018-8 and Fapemig Grant PPM-00450-17.}\,\, and  Luis Alberto Garcia\thanks{Supported by a CAPES scholarship.\newline 2020 Mathematics Subject Classification: 46B42, 46G25, 47B65, 47H60.\newline Keywords: Banach lattices, Aron-Berner extension, almost Dunford-Pettis multilinear operators, separately almost Dunford-Pettis operators.
}}
\date{}
\maketitle

\begin{abstract} We prove several results establishing conditions on the Banach lattices $E_1, \ldots, E_m$ and $F$ so that the Aron-Berner extensions of (separately) almost Dunford-Pettis $m$-linear operators from $E_1 \times \cdots \times E_m$ to $F$ are (separately) almost Dunford-Pettis. Illustrative examples are provided.  

\end{abstract}

\section{Introduction} Aron-Berner extensions, or Arens extensions, of multilinear operators between Banach spaces have been studied by many authors for decades (see \cite{dineen} and for recent developments see, e.g., \cite{daniela, cambridge, pilarenrique, elisa}). Aron-Berner extensions of multilinear mappings between Riesz spaces and Banach lattices we treated in, e.g., \cite{boulabiarbuskespage, ryanpositivity, buskesroberts}. A natural trend in this area is to investigate if a given property of multilinear operators passes from the operator to its extensions. In the realm of Banach spaces, the case of weakly sequentially continuous multilinear operators was studied in, e.g., \cite{gutierrez, zalduendo}. Inspired by the notion of almost Dunford-Pettis linear operator, in this paper we study the lattice counterpart of the class of weakly sequentially continuous multilinear operators and their Aron-Berner extensions.

Recall that a linear operator $u$ from a Banach lattice $E$ to a Banach space $G$ is almost Dunford-Pettis, see, e.g., \cite{bel, wic, gebu}, if $u$ sends disjoint weakly null sequences in $E$ to norm null sequences in $G$. The following notions are quite natural: for Banach lattices $E_1, \ldots, E_m$ and a Banach space $G$, an $m$-linear operator $A \colon E_1 \times \cdots \times E_m \longrightarrow G$ is said to be:\\
$\bullet$ {\it Separately almost Dunford-Pettis} if for any $j = 1, \ldots, m$, and all $x_1 \in E_1, \ldots, x_{j-1} \in E_{j-1}, x_{j+1} \in E_{j+1}, \ldots, x_m \in E_m$, the map
$$x_j \in E_j \longrightarrow A(x_1, \ldots, x_m) \in G, $$
is an almost Dunford-Pettis linear operator. The operator above shall be denoted by $A(x_{1},\ldots,x_{j-1},\bullet,x_{j+1},\ldots,x_{m})$.\\
$\bullet$ {\it Almost Dunford-Pettis} if the sequence $(A(x_{1,n}, \ldots, x_{m,n}))_{n=1}^\infty$ is norm null in $G$ whenever $(x_{j,n})_{n=1}^\infty$ is a disjoint weakly null sequence in $E_j$, $j = 1, \ldots, m$.

The aim of this paper is to study when Aron-Berner extensions of (separately) almost Dunford-Pettis multilinear operators are (separately) almost Dunford-Pettis. In Section \ref{sec2} we first prove some results we believe will convince the  reader that it is quite rare for Aron-Berner extensions of separately almost Dunford-Pettis operators to be separately almost Dunford-Pettis. Nevertheless, next we prove some positive results in this direction. In particular we show that Aron-Berner extensions of multilinear operators of finite rank are separately almost Dunford-Pettis and that Aron-Berner extensions of positive $m$-linear operators from $c_0^m$ to $\ell_1$ are separately almost Dunford-Pettis. We start Section \ref{sec3} comparing almost Dunford-Pettis with separately almost Dunford-Pettis operators. Then we prove some results that will enable us to give examples of almost Dunford-Pettis multilinear operators whose Aron-Berner extensions are not almost Dunford-Pettis. Next we prove several positive results, mainly for multilinear operators on Banach lattices whose duals have the positive Schur property. Examples are provided. 

By $G^*$ we denote the dual of the Banach space (or Banach lattice) $G$, by $G^{**}$ its bidual and by $B_G$ its closed unit ball. By $J_G \colon G \longrightarrow G^{**}$ we mean the canonical embedding. As usual, the space of bounded linear operators from $E$ to $F$ is denoted by ${\cal L}(E;F)$.

Now we give the description of the Aron-Berner extensions that shall serve our purposes. Let $E_1, \ldots, E_m,F$ be real Banach spaces and ${\cal L}(E_1, \ldots, E_m;F)$ be the space of continuous $m$-linear operators from $E_1 \times \cdots \times E_m$ to $F$. When $F$ is the scalar field we write ${\cal L}(E_1, \ldots, E_m)$. For the spaces of regular $m$-linear operators/forms we write ${\cal L}_r$ instead of ${\cal L}$. $S_m$ stands for the set of permutations of $\{1, \ldots, m\}$. For $\rho\in S_{m}$ and $k\in \{1,\ldots,m\}$, we fix the following notation:
$$E_{1},\ldots,\,_{\rho(1)}E,\ldots,\,_{\rho(k-1)}E,\ldots,E_{m}=\left\{ \begin{array}{cl}
E_{1},\ldots, E_{m} \mbox{~in this order} & \mbox{if}\,\ k=1, \\
E_{1},\ldots, E_{m} \mbox{~in this order, where}\\ E_{\rho(1)},\ldots,
   E_{\rho(k-1)} \mbox{~are removed} & \mbox{if}~ k=2,\ldots,m.
\end{array}\right.$$
For instance, $(E_1,\,_2E, E_3) = (E_1, E_3)$.
The same procedure defines the $(m-k+1)$-tuple $(x_{1},\ldots,\,_{\rho(1)}x,\ldots,\,_{\rho(k-1)}x,\ldots,x_{m})$ and the cartesian product $E_{1}\times\cdots \times\,_{\rho(1)}E \times\cdots \times\,_{\rho(k-1)}E\times\cdots \times E_{m}$.   Moreover, for $k=1,\ldots,m-1$, we write $$E_{1},\ldots,\,_{\rho(1)}E,\ldots,\,_{\rho(k)}E,\ldots,E_{m}=E_{1},\ldots, E_{m}$$ in this order, where $E_{\rho(1)},\ldots,E_{\rho(k)}$ are removed. In the same fashion we define the $(m-k)$-tuple $(x_{1},\ldots,_{\rho(1)}x,\ldots,_{\rho(k)}x,\ldots,x_{m})$ and the corresponding cartesian product. Finally, for $k=m$ we write $\mathcal{L}(E_{1},\ldots,\,_{\rho(1)}E,\ldots,\,_{\rho(k)}E,\ldots,E_{m};\mathbb{R})=\mathbb{R}.$


Given a permutation $\rho\in S_{m}$ and a continuous $m$-linear operator $A \colon E_{1}\times\cdots\times E_{m} \longrightarrow F$, the Aron-Berner extension of $A$ with respect to $\rho$ is the $m$-linear operator $AB_{m}^{\rho}(A)\colon E_{1}^{**}\times
\cdots\times E_{m}^{**}\longrightarrow F^{**}$ defined by $$ AB_{m}^{\rho}(A)(x_{1}^{**},\ldots,x_{m}^{**})(y^{*})=\big(\overline{x_{
\rho(m)}^{**}}^{\rho}\circ\cdots\circ\overline
{x_{\rho(1)}^{**}}^{\rho}\big)(y^{*}\circ A),$$
where, for every $x_{\rho(k)}^{**} \in E_{\rho(k)}^{**}$, the linear operator
\begin{equation}\label{defope}\overline{x_{\rho(k)}^{\prime\prime}}^{\rho}\colon \mathcal{L}(E_{1},\ldots,\,_{\rho(1)}E,\ldots,\,_{\rho(k-1)}E,\ldots,E_{m})\longrightarrow \mathcal{L}(E_{1},\ldots,\,_{\rho(1)}E,\ldots,\,_{\rho(k)}E,\ldots,E_{m}),
\end{equation}
is given by
$$\overline{x_{\rho(k)}^{\prime\prime}}^{\rho}(A)(x_{1},\ldots,\,_{\rho(1)}x,
\ldots,\,_{\rho(k)}x,\ldots,x_{m})=x_{\rho(k)}^{\prime\prime}(A(x_{1},\ldots,\,_{\rho(1)}x,\ldots,\,_{\rho(k)}x,\bullet\,,\ldots,x_{m})).$$
Details can be found in \cite{lg, villanueva}.


\section{Separately almost Dunford-Pettis operators}\label{sec2}

A bidual extension of a continuous $m$-linear operator $A\colon G_{1}\times\cdots\times G_{m}\longrightarrow G$ between Banach spaces is a continuous $m$-linear operator  $\widetilde{A}\colon G_{1}^{\ast\ast}\times\cdots\times G_{m}^{\ast\ast}\longrightarrow G^{\ast\ast}$ such that $\widetilde{A} \circ (J_{G_{1}},\ldots,J_{G_{m}})=J_{G} \circ A$. Of course, Aron-Berner extensions are bidual extensions.

Using that the canonical embedding $J_E \colon E \longrightarrow E^{**}$ is a Riesz homomorphism for every Banach lattice $E$, it is immediate that only bidual extensions of separately almost Dunford-Pettis operators can be separately almost Dunford-Pettis.

Let $E_1, \ldots, E_m$ be Banach lattices. Since bounded linear functionals on Banach lattices are almost Dunford-Pettis, continuous $m$-linear forms $A \colon E_1 \times \cdots \times E_m\longrightarrow \mathbb{R}$ and all their bidual extensions are separately almost Dunford-Pettis. Let us see that, in the vector-valued case, it may happen that all Aron-Berner extensions of a separately almost Dunford-Pettis multilinear operator are not separately almost Dunford-Pettis.



\begin{example}\rm Recall that a Banach lattice $E$ has the positive Schur property if disjoint (or positive or disjoint positive) weakly null sequences in $E$ are norm null (see, e.g., \cite{wnukbari}). Consider the Banach lattice $E = \big(\oplus_{n}\ell_{\infty}^{n}\big)_{1}$, where each $\ell^n_\infty$ is $\mathbb{R}^n$ endowed with the sup norm and with the coordinatewise order. It is well known that $E$ has the positive Schur property and in \cite[Example 2.8]{jg} is it proved that $E^{**}$ lacks the positive Schur property. The positive bilinear operator
$$A\colon\mathbb{R}\times E \longrightarrow E~, A(\lambda,y)=\lambda y,$$
is separately almost Dunford-Pettis: It is obvious that $A(\bullet, y)$ is almost Dunford-Pettis for every $y \in E$, and $A(\lambda, \bullet)$ is almost Dunford-Pettis for every $\lambda \in \mathbb{R}$ because $E$ has the positive Schur property. Since linear operators from $\mathbb{R}$ to $E^*$ and from $E$ to $\mathbb{R}^* = \mathbb{R}$ are compact, \cite[Theorem 1]{bovi} implies that $A$ admits a unique bidual extension $\widetilde{A} \colon \mathbb{R} \times E^{**} \longrightarrow E^{**}$, which is separately weak$^*$-weak$^{*}$ continuous. But it is well known that if a bilinear operator admits a separately weak$^*$-weak$^{*}$ continuous bidual extension $B$, then all its Aron-Berner extensions coincide with $B$ (see \cite{peralta}). So, the only Aron-Berner extension of $A$ is $\widetilde{A}$. Combining Goldstine's Theorem with the weak$^*$-weak$^{*}$ continuity of $\widetilde{A}$ it follows that
$$\widetilde{A}(\lambda,y^{\ast\ast}) = \lambda y^{**} \mbox{ for all } \lambda \in \mathbb{R} \mbox{ and } y^{**} \in E^{**}. $$
Thus, $\widetilde{A}(1,\bullet)$ is the identity operator on ${E}^{\ast\ast}$, which is not almost Dunford-Pettis because ${E}^{\ast\ast}$ fails the positive Schur property. Therefore $\widetilde{A}$, which is the only Aron-Berner extension of $A$, is not separately almost Dunford-Pettis.
\end{example}

\begin{proposition} Let $m \in \mathbb{N}$. The following are equivalent for the Banach lattices $E_{1},\ldots,E_{m}$ and $F$. \\
{\rm (a)} If $j \in \{1, \ldots, m\}$ and the linear operator $T \colon E_{j} \longrightarrow F$ is almost Dunford-Pettis, then its second adjoint  $T^{\ast\ast}\colon E_{j}^{\ast\ast} \longrightarrow F^{\ast\ast}$ is almost Dunford-Pettis.\\
 {\rm (b)} If $A \colon E_1 \times \cdots \times E_m \longrightarrow F $ is a separately almost Dunford-Pettis $m$-linear operator which admits a separately weak$^{*}$-weak$^{*}$ continuous extension $\widetilde{A} \colon {E_1}^{\ast\ast} \times \cdots \times E_{m}^{\ast\ast} \longrightarrow F^{\ast\ast}$, then for every $j = 1, \ldots, m$, and all $x_i \in E_i$, $i \neq j$, the linear operator  $\widetilde{A}_{x_{1},\ldots,x_{j-1},x_{j+1},\ldots,x_{m}}\colon E_{j}^{\ast\ast}\longrightarrow F^{\ast\ast}$ given by $$\widetilde{A}_{x_{1},\ldots,x_{j-1},x_{j+1},\ldots,x_{m}}(x_{j}^{\ast\ast})=\widetilde{A}(J_{E_{1}}(x_{1}),\ldots,J_{E_{j-1}}(x_{j-1}),x_{j}^{\ast\ast},J_{E_{j+1}}(x_{j+1}),\ldots,J_{E_{m}}(x_{m})),$$ is almost Dunford-Pettis.
\end{proposition}

\begin{proof} (a)$\Rightarrow$(b) Let $A$ be a separately almost Dunford-Pettis $m$-linear operator which admits a separately weak$^{*}$-weak$^{*}$ continuous extension $\widetilde{A}$, $j\in\{1,\ldots,m\}$ and  $x_{i}\in E_{i}$ for $i\neq j$. The linear operator $T\colon E_{j}\longrightarrow F$ given by $T(x_{j})=A(x_{1},\ldots,x_{m})$ is almost Dunford-Pettis because $A$ is separately almost Dunford-Pettis. So, $T^{**}$ is almost Dunford-Pettis by assumption. Let $(x_{j,n}^{\ast\ast})_{n=1}^{\infty}$ be a disjoint positive weakly null sequence in $E_{j}^{\ast\ast}$.  We have $\|T^{\ast\ast}(x_{j,n}^{\ast\ast})\|\longrightarrow 0$. 
For every $n$, by Goldstine's Theorem there is a net $(x^n_{\alpha_{j}})_{\alpha_{j}\in\Omega_{j}}$ in $E_{j}$ such that $x_{j,n}^{\ast\ast}=\omega^{\ast}$-$\displaystyle\lim_{\alpha_{j}}J_{E_{j}}(x^n_{\alpha_j})$. 
Since $\widetilde{A}$ is separately weak$^{*}$-weak$^{*}$ continuous,
\begin{align*}
\widetilde{A}_{x_{1},\ldots,x_{j-1},x_{j+1},\ldots,x_{m}}(x_{j,n}^{\ast\ast})&=\widetilde{A}(J_{E_{1}}(x_{1}),\ldots,J_{E_{j-1}}(x_{j-1}),x_{j,n}^{\ast\ast},J_{E_{j+1}}(x_{j+1}),\ldots,J_{E_{m}}(x_{m}))\\
&=\omega^{\ast}{\rm -}\lim_{\alpha_{j}}\widetilde{A}(J_{E_{1}}(x_{1}),\ldots,J_{E_{j}}(x^n_{\alpha_j}),\ldots,J_{E_{m}}(x_{m}))\\
&=\omega^{\ast}{\rm -}\lim_{\alpha_{j}}J_{F}(A(x_{1},\ldots,x^n_{\alpha_{j}},\ldots,x_{m}))\\
& =\omega^{\ast}{\rm -}\lim_{\alpha_{j}} J_{F}(T(x^n_{\alpha_{j}}))=\omega^{\ast}{\rm -}\lim_{\alpha_{j}}T^{\ast\ast}(J_{E_{j}}(x^n_{\alpha_{j}}))=T^{\ast\ast}(x_{j,n}^{\ast\ast})
\end{align*}
for every $n$. Making $n \to \infty$ we get  $\|\widetilde{A}_{x_{1},\ldots,x_{j-1},x_{j+1},\ldots,x_{m}}(x_{j,n}^{\ast\ast})\|=\|T^{\ast\ast}(x_{j,n}^{\ast\ast})\|\longrightarrow 0$, showing that $\widetilde{A}_{x_{1},\ldots,x_{j-1},x_{j+1},\ldots,x_{m}}$ is almost Dunford-Pettis.\\
(b)$\Rightarrow$(a) Let $i \in \{1, \ldots, m\}$ and $T\colon E_{i} \longrightarrow F$ be an almost Dunford-Pettis linear operator. Choose $0 \neq \varphi_{j}\in E_{j}^{\ast}$ for $j\neq i$ and consider the $m$-linear operator
$$A\colon E_{1}\times\cdots \times E_{m}\longrightarrow F~,~A(x_{1},\ldots,x_{m})=T(x_{i})\cdot \textstyle\prod\limits_{j=1, j\neq i}^{m}\varphi_{j}(x_{j}).$$ 
Since the operator $T$ and the functionals $\varphi_{j}$, $j \neq i$, are almost Dunford-Pettis, $A$ is separately almost Dunford-Pettis. Since adjoint operators are weak$^{*}$-weak$^{*}$ continuous, it is immediate that the $m$-linear operator
  $$\widetilde{A} \colon E_{1}^{\ast\ast} \times \cdots \times E_{m}^{\ast\ast}\longrightarrow F^{\ast\ast}~,~ \widetilde{A}(x_{1}^{\ast\ast},\ldots,x_{m}^{\ast\ast}) = T^{\ast\ast}(x_{i}^{\ast\ast})\cdot \textstyle\prod\limits_{j=1, j\neq i}^{m}x_{j}^{\ast\ast}(\varphi_{j}),$$
is a separately weak$^{*}$-weak$^{*}$ continuous bidual extension of $A$. So $A$ fulfills the assumptions of (b). 
Choosing $x_{j}\in E_{j}$ such that $\varphi_{j}(x_{j})=1$ for each $j \neq i$, by (b) the operator $\widetilde{A}_{x_{1},\ldots,x_{j-1},x_{j+1},\ldots,x_{m}}$ is almost Dunford-Pettis. For every $x_{i}^{\ast\ast} \in E_i^{**}$, $$\widetilde{A}_{x_{1},\ldots,x_{j-1},x_{j+1},\ldots,x_{m}}(x_i^{**})=\widetilde{A}(J_{E_{1}}(x_{1}),\ldots,x_{i}^{\ast\ast},\ldots,J_{E_{m}}(x_{m}))=T^{\ast\ast}(x_{i}^{\ast\ast}),$$
from which it follows that 
$T^{\ast\ast}$ is almost Dunford-Pettis.
\end{proof}

As to conditions (a) and (b) of the Proposition above, it is worth emphasizing the following:\\
$\bullet$ Condition (a) was investigated in \cite{lg2} and, as the reader can check there, it holds only under strong assumptions on the spaces and/or on the operators.\\
$\bullet$ Condition (b), which is equivalent to (a), is much less than the desired implication, which is: $A$ is separately almost Dunford-Pettis $ \Longrightarrow$ the Aron-Berner extensions of $ A$ are separately almost Dunford-Pettis.

This means that, even under strong assumptions, only a little can be obtained. The conclusion is that the desired implication should hold only in very specific situations. For the rest of this section we shall pursue such situations.

The first situation is the case of finite rank multilinear operators.
Recall that a map taking values in a linear space has finite rank if the subspace generated by its range is finite dimensional. Using that multilinear forms are separately almost Dunford-Pettis, it is easy to see that if $E_1, \ldots, E_m$ are Banach lattices and $G$ is a Banach space, then every continuous $m$-linear operator from $E_1 \times \cdots \times E_m$ to $G$ of finite rank is separately almost Dunford-Pettis.

\begin{proposition}\label{propo1}
All Aron-Berner extension of a finite rank $m$-linear operator $A\in\mathcal{L}(E_{1},\ldots,E_{m};G)$ are of finite rank, hence separately almost Dunford-Pettis. 
\end{proposition}
\begin{proof}
Given a finite rank operator $A\in\mathcal{L}(E_{1},\ldots,E_{m};G)$, there are  $k\in\mathbb{N}$, $B_{j}\in\mathcal{L}(E_{1},\ldots,E_{m})$ and $y_{j}\in G, j=1,\ldots,k$, such that $A = \sum\limits_{j=1}^k B_j \otimes y_j$, meaning that $A(x_{1},\ldots,x_{m})=\sum\limits_{j=1}^{k}B_{j}(x_{1},\ldots,x_{m})y_j$. 
For $\rho\in S_{m}$, $x_{\rho(1)}^{\ast\ast}\in E_{\rho(1)}^{\ast\ast}$ and $y^{\ast}\in G^{\ast}$, putting $C_{j}=y^{\ast}\circ  (B_{j}\otimes y_{j})$ we get
\begin{align*}
\overline{x_{\rho(1)}^{\ast\ast}}^{\rho}&(C_{j})(x_{1},\ldots,_{\rho(1)}x,\ldots,x_{m})=x_{\rho(1)}^{\ast\ast}(C_{j}(x_{1},\ldots,_{\rho(1)}x,\bullet,\ldots,x_{m}))\\
&=
x_{\rho(1)}^{\ast\ast}(y^{\ast}(y_{j}) B_{j}(x_{1},\ldots,_{\rho(1)}x,\bullet,\ldots,x_{m}))\\
&=y^{\ast}(y_{j}) x_{\rho(1)}^{\ast\ast}(B_{j}(x_{1},\ldots,_{\rho(1)}x,\bullet,\ldots,x_{m}))=y^{\ast}(y_{j}) \overline{x_{\rho(1)}^{\ast\ast}}^{\rho}(B_{j})(x_{1},\ldots,_{\rho(1)}x,\ldots,x_{m}),
\end{align*}
that is, $\overline{x_{\rho(1)}^{\ast\ast}}^{\rho}(C_{j})=y^{\ast}(y_{j}) \overline{x_{\rho(1)}^{\ast\ast}}^{\rho}(B_{j})$. Since $y^{\ast}\circ A= 
\sum\limits_{j=1}^{k}C_{j}$, for $x_{\rho(i)}^{\ast\ast}\in E_{\rho(i)}^{\ast\ast}, i=2,\ldots,m$,
\begin{align*}
AB&_{m}^{\rho}(A)(x_{1}^{\ast\ast},\ldots,x_{m}^{\ast\ast})(y^{\ast})=\big(\overline{x_{\rho(m)}^{\ast\ast}}^{\rho}\circ\cdots\circ \overline{x_{\rho(1)}^{\ast\ast}}^{\rho}\big)(y^{\ast}\circ A)\\
&=\big(\overline{x_{\rho(m)}^{\ast\ast}}^{\rho}\circ\cdots\circ \overline{x_{\rho(2)}^{\ast\ast}}^{\rho}\big)\Big(\sum_{j=1}^{k}\overline{x_{\rho(1)}^{\ast\ast}}^{\rho}(C_{j})\Big)=\displaystyle\sum_{j=1}^{k}y^{\ast}(y_{j}) \big(\overline{x_{\rho(m)}^{\ast\ast}}^{\rho}\circ\cdots\circ \overline{x_{\rho(2)}^{\ast\ast}}^{\rho}\big)\big(\overline{x_{\rho(1)}^{\ast\ast}}^{\rho}(B_{j})\big)\\
&=\displaystyle\sum_{j=1}^{k}y^{\ast}(y_{j})AB_{m}^{\rho}(B_{j})(x_{1}^{\ast\ast},\ldots,x_{m}^{\ast\ast})=\displaystyle\sum_{j=1}^{k}AB_{m}^{\rho}(B_{j})\otimes J_{F}(y_{j})(x_{1}^{\ast\ast},\ldots,x_{m}^{\ast\ast})(y^{\ast}).
\end{align*}
It follows that $AB_{m}^{\rho}(A)= \sum\limits_{j=1}^{k}AB_{m}^{\rho}(B_{j})\otimes J_{F}(y_{j})$ is a finite rank operator. 
\end{proof}

Now we proceed to the second situation. Making the identification $J_F(F) = F$, a bidual extension of an $F$-valued linear or multilinear operator is {\it genuine} if it is $F$-valued as well. For example, the second adjoint $T^{**}$ of a weakly compact linear operator $T$ is a genuine bidual extension of $T$. Genuine Aron-Berner extensions of multilinear operators have already been studied in the literature (see \cite{bovi, peralta}).

\begin{proposition}\label{propo6}
Let $E_{1},\ldots,E_{m}$ be Banach lattices whose duals have the positive Schur property and let  $F$ be a Banach lattice with order continuous norm. If a positive $m$-linear operator $A \colon E_{1} \times \cdots \times E_{m} \longrightarrow F$ admits a positive genuine separately weak$^*$-weak$^*$ continuous bidual extension $\widetilde{A} \colon E_{1}^{\ast\ast} \times \cdots \times E_{m}^{\ast\ast} \longrightarrow F^{\ast\ast}$, then  $\widetilde{A}$ is separately almost Dunford-Pettis. In particular, $A$ is separately almost Dunford-Pettis.
\end{proposition}

\begin{proof} Let $i\in\{1,\ldots,m\}$ and positive functionals $z_{j}^{\ast\ast}\in E_{j}^{\ast\ast}, j=1,\ldots,m$, $j\neq i$, be given. Considering the inverse  $J_{F}^{-1}\colon J_{F}(F)\longrightarrow F$ of the canonical embedding, consider the positive linear operator
$$T\colon E_{i}\longrightarrow F~,~T(x_{i})=J_{F}^{-1}(\widetilde{A}(z_{1}^{\ast\ast},\ldots,z_{i-1}^{\ast\ast},J_{E_{i}}(x_{i}),z_{i+1}^{\ast\ast},\ldots,z_{m}^{\ast\ast})),$$
which is well defined because the bidual extension $\widetilde{A}$ is genuine. 
 Given $x_{i}^{\ast\ast}\in E_{i}^{\ast\ast}$ and a net $(x_{\alpha_{i}})_{\alpha_{i}\in \Omega_{i}}$ in $E_{i}$ such that $x_{i}^{\ast\ast}=\omega^{\ast}$-$\displaystyle\lim_{\alpha_{i}}J_{E_{i}}(x_{\alpha_{i}})$, 
 the separate weak$^*$-weak$^*$ continuity of $\widetilde{A}$ gives
\begin{align*}
T^{\ast\ast}(x_{i}^{\ast\ast})&=\omega^{\ast}{\rm -}\displaystyle\lim_{\alpha_{i}} T^{\ast\ast}(J_{E_{i}}(x_{\alpha_{i}}))=\omega^{\ast}{\rm -}\displaystyle\lim_{\alpha_{i}} (T^{\ast\ast}\circ J_{E_{i}})(x_{\alpha_{i}})\\
&=\omega^{\ast}{\rm -}\displaystyle\lim_{\alpha_{i}} J_{F}(T(x_{\alpha_{i}}))=\omega^{\ast}{\rm }-\displaystyle\lim_{\alpha_{i}}\widetilde{A}(z_{1}^{\ast\ast},\ldots,z_{i-1}^{\ast\ast},J_{E_{i}}(x_{\alpha_{i}}),z_{i+1}^{\ast\ast},\ldots,z_{m}^{\ast\ast})\\
&=\widetilde{A}(z_{1}^{\ast\ast},\ldots,z_{i-1}^{\ast\ast},x_{i}^{\ast\ast},z_{i+1}^{\ast\ast},\ldots,z_{m}^{\ast\ast})\subseteq J_{F}(F),
\end{align*}
showing that $T$ is weakly compact, hence $T^{\ast\ast}$ is almost Dunford-Pettis by \cite[Proposition 5.5]{lg2}. So, given a disjoint weakly null sequence $(x_{i,n}^{\ast\ast})_{n=1}^{\infty}$ in $E_{i}^{\ast\ast}$ we have 
\begin{equation}\label{rt6y}\|\widetilde{A}(z_{1}^{\ast\ast},\ldots,z_{i-1}^{\ast\ast},x_{i,n}^{\ast\ast},z_{i+1}^{\ast\ast},\ldots,z_{m}^{\ast\ast})\|=\|T^{\ast\ast}(x_{i,n }^{\ast\ast})\|\longrightarrow 0.\end{equation}
For given $x_{j}^{\ast\ast}\in E_{j}^{\ast\ast}, j=1,\ldots,m$, $j\neq i$, write $x_{j}^{\ast\ast}=(x_{j}^{\ast\ast})^{+}-(x_{j}^{\ast\ast})^{-}$. For each  $\xi^{m}=(\xi_{1},\ldots,\xi_{i-1},\xi_{i+1},\ldots,\xi_{m}) \in \{+1, -1\}^{m-1}$, put $P(\xi^{m})=\xi_{1}\cdots\xi_{i-1}\xi_{i+1}\cdots\xi_{m}\in \{+1,-1\}$. Since $(x_{j}^{\ast\ast})^{+}$ and $(x_{j}^{\ast\ast})^{-}$ are positive functionals, by (\ref{rt6y}) we get
\begin{align*}
&\|\widetilde{A}(x_{1}^{\ast\ast},\ldots,x_{i-1}^{\ast\ast},x_{i,n}^{\ast\ast},x_{i+1}^{\ast\ast},\ldots,x_{m}^{\ast\ast})\|\\
&=\|\widetilde{A}((x_{1}^{\ast\ast})^{+}\hspace{-0.1cm}\hspace{-0.1cm}-(x_{1}^{\ast\ast})^{-},\ldots,(x_{i-1}^{\ast\ast})^{+}\hspace{-0.1cm}-(x_{i-1}^{\ast\ast})^{-},x_{i,n}^{\ast\ast},(x_{i+1}^{\ast\ast})^{+}\hspace{-0.1cm}-(x_{i+1}^{\ast\ast})^{-},\ldots,(x_{m}^{\ast\ast})^{+}\hspace{-0.1cm}-(x_{m}^{\ast\ast})^{-})\|\\
&=\Big\|\sum_{P(\xi^{m})=+1} \widetilde{A}((x_{1}^{\ast\ast})^{\xi_{1}},\ldots,(x_{i-1}^{\ast\ast})^{\xi_{i-1}},x_{i,n}^{\ast\ast},(x_{i+1}^{\ast\ast})^{\xi_{i+1}},\ldots,(x_{m}^{\ast\ast})^{\xi_{m}})\\
&\quad\quad-\sum_{P(\xi^{m})=-1} \widetilde{A}((x_{1}^{\ast\ast})^{\xi_{1}},\ldots,(x_{i-1}^{\ast\ast})^{\xi_{i-1}},x_{i,n}^{\ast\ast},(x_{i+1}^{\ast\ast})^{\xi_{i+1}},\ldots,(x_{m}^{\ast\ast})^{\xi_{m}})\Big\|\\
&\leq\sum_{P(\xi^{m})=+1} \|\widetilde{A}((x_{1}^{\ast\ast})^{\xi_{1}},\ldots,(x_{i-1}^{\ast\ast})^{\xi_{i-1}},x_{i,n}^{\ast\ast},(x_{i+1}^{\ast\ast})^{\xi_{i+1}},\ldots,(x_{m}^{\ast\ast})^{\xi_{m}})\|\\
&\quad\quad +\sum_{P(\xi^{m})=-1} \|\widetilde{A}((x_{1}^{\ast\ast})^{\xi_{1}},\ldots,(x_{i-1}^{\ast\ast})^{\xi_{i-1}},x_{i,n}^{\ast\ast},(x_{i+1}^{\ast\ast})^{\xi_{i+1}},\ldots,(x_{m}^{\ast\ast})^{\xi_{m}})\|\longrightarrow 0.
\end{align*}
This proves that 
$\widetilde{A}$ is separately almost Dunford-Pettis.
\end{proof}

In order to give concrete applications of Proposition \ref{propo6}, recall that a Banach space $E$ is {\it Arens regular} if every bounded linear operator from $E$ to $E^*$ is weakly compact.

\begin{corollary}\label{corsepadp} Let $E, F$ be Banach lattices such that $E$ is Arens-regular, $E^{\ast}$ has the positive Schur property, $F$ has order continuous norm and every bounded linear operator from $E$ to $F$ is weakly compact. Then, for every $m \in \mathbb{N}$, the Aron-Berner extensions of any positive $m$-linear operator from $E^m$ to  $F$ are separately almost 
Dunford-Pettis.
\end{corollary}

\begin{proof} Let $A \colon E^m \longrightarrow F$ be a positive $m$-linear operator. It is well known that the Arens regularity of $E$ gives that all Aron-Berner extensions of $A$ coincide and are separately weak$^*$-weak$^*$ continuous. It is also well known that if every operator from $E$ to $F$ is weakly compact, then such extension is genuine (all this information can be found in \cite{bovi, peralta}). Now the result follows from Proposition \ref{propo6}.
\end{proof}

\begin{example}\rm Bounded linear operators from $c_0$ to $\ell_1$ are compact (Pitt's Theorem), hence weakly compact, so $c_0$ is Arens regular. It is clear that $c_0^* =\ell_1$ has the positive Schur property, so Corollary \ref{corsepadp} assures that, for every $m$, the Aron-Berner extensions of every positive $m$-linear operator $A \colon c_0^m \longrightarrow \ell_1$ are separately almost Dunford-Pettis.
\end{example}

\section{Almost Dunford-Pettis multilinear operators}\label{sec3}

Unless stated otherwise, throughout this section $E_{1},\ldots, E_{m}, F$ are Banach lattices and $G$ is a Banach space. From the definition (cf. Introduction), it follows easily that an $m$-linear operator $A\colon E_{1}\times\cdots\times E_{m}\longrightarrow G$ is almost Dunford-Pettis if and only if $(A(x_{1,n}, \ldots, x_{m,n}))_{n=1}^\infty$ is norm null in $G$ whenever $(x_{j,n})_{n=1}^\infty$ is a {\it positive} disjoint weakly null sequence in $E_j$, $j = 1, \ldots, m$.

We start by showing that none of the following implications holds in general:
$$A \mbox{ is separately almost Dunford-Pettis} \Longleftrightarrow A \mbox{ is  almost Dunford-Pettis}. $$

\begin{example}\rm On the one hand, for $1< p < \infty$, the positive bilinear form $A\colon \ell_{p}\times \ell_{p^{\ast}}\longrightarrow \mathbb{R}$, $ A(x,y)=\sum\limits_{n=1}^{\infty}x_{n}y_{n}$, is obvious separately almost Dunford-Pettis, but it is not almost Dunford-Pettis:   $e_{n}\xrightarrow{\,\, \omega\,\,} 0$ in $\ell_{p}$ and in $\ell_{p^{\ast}}$ and $A(e_{n},e_{n})=1$ for every $n\in\mathbb{N}$.

 On the other hand, the positive bilinear operator $A\colon \mathbb{R}\times \ell_{2}\longrightarrow \ell_{\infty}$, $ A(\lambda,x)=\lambda x$, is obviously almost Dunford-Pettis, but it is not separately almost Dunford-Pettis: $e_{n}\xrightarrow{\,\, \omega\,\,} 0$ in $\ell_{2}$ but $\|A(1,e_n)\| = 1$ for every $n$.

As the operators in this example are defined on reflexive spaces, Aron-Berner extensions of separately almost Dunford-Pettis are not always almost Dunford-Pettis and vice-versa.
\end{example}

          Among other consequences, the next lemma establishes that only Aron-Berner extensions of almost Dunford-Pettis operators can be almost Dunford-Pettis and will enable us to show that Aron-Berner extensions of almost Dunford-Pettis operators are not always almost Dunford-Pettis.

\begin{lemma}\label{le2}
 Let $A\colon E_{1}\times\cdots\times E_{m}\longrightarrow G$ be a continuous $m$-linear operator.\\
{\rm (a)} If $A$ admits an almost Dunford-Pettis bidual extension, then $A$ is almost Dunford Pettis as well.\\
{\rm (b)} If some of the Banach lattices $E_{1},\ldots, E_{m}$ has the positive Schur property, then $A$ is almost Dunford-Pettis.
\end{lemma}

\begin{proof} (a) Let $\widetilde{A}$ be an almost Dunford-Pettis bidual extension of $A$. The result follows easily from the equality $\widetilde{A} \circ (J_{E_1}, \ldots, J_{E_m}) = J_G \circ A$ and the fact that the canonical operators $J_{E_i}$ are Riesz homomorphisms and $J_G$ is an isometric embedding. \\
(b) Let $j\in\{1,\ldots,m\}$ be such that $E_{j}$ has the positive Schur property and let $(x_{i,n})_{n=1}^{\infty}$ be disjoint positive weakly null sequences in  $E_{i}$, $i = 1,\ldots,m$. Then these sequences are bounded, say $\|x_{i,n}\|\leq M_i$ for every $n$, and $\|x_{j,n}\|\longrightarrow 0$. Therefore, 
$$\|A(x_{1,n},\ldots,x_{m,n})\|\leq  M_{1}\cdots M_{j-1}M_{j+1}\cdots M_{m}\cdot \|A\|\cdot\|x_{j,n}\|\longrightarrow 0.$$
\end{proof}

\begin{theorem}\label{propo3}
 Let $E_{1},\ldots,E_{m}, F$ be Banach lattices with $F$ $\sigma$-Dedekind complete containing a copy of $\ell_{\infty}$. The following are equivalent.\\
{\rm (a)} Every continuous $m$-linear operator $A \colon E_{1} \times \cdots \times E_{m} \longrightarrow F$ admits an almost Dunford-Pettis Aron-Berner extension.\\
{\rm (b)} Every regular $m$-linear operator $A \colon E_{1} \times \cdots \times E_{m} \longrightarrow F$ admits an almost Dunford-Pettis Aron-Berner extension.\\
{\rm (c)} Every positive $m$-linear operator $A \colon E_{1} \times \cdots \times E_{m} \longrightarrow F$ admits an almost Dunford-Pettis Aron-Berner extension.\\
{\rm (d)} $E_{j}^{\ast\ast}$ has the positive Schur property for some $j\in\{1,\ldots,m\}$.
\end{theorem}

\begin{proof} (a)$\Rightarrow$(b)$\Rightarrow$(c) are trivial and (d)$\Rightarrow$(a) follows from Lemma \ref{le2}(b). Let us prove (c)$\Rightarrow$(d). Suppose that $E_{1}^{\ast\ast}, \ldots, E_{m}^{\ast\ast}$ fail the positive Schur property. By \cite[Proposition 2.1]{wic}, for each $j=1,\ldots,m$, there is a positive disjoint weakly null normalized sequence $(x_{j,n}^{\ast\ast})_{n=1}^{\infty}$ in $E_{j}^{\ast\ast}$. 
Take  positive sequences $(x_{j,n}^{\ast})_{j=1}^{\infty}$ in $B_{E_{j}^{\ast}}$ such that que $x_{j,n}^{\ast\ast}(x_{j,n}^{\ast})\geq \frac{1}{2}$ for all $j = 1, \ldots, m$ and $n \in \mathbb{N}$. From \cite[Theorem 4.56]{positiveoperators} we know that the norm of $F$ is not order continuous and from \cite[Theorem 2.4.2]{nieberg} that there are $y\in F$  and a disjoint sequence  $(y_{n})_{n=1}^{\infty}$ in $F$ such that $0\leq y_{n}\leq y$ and $\|y_{n}\|=1$ for every $n$. For  $x_{j}\in E_{j}^{+}$ and $k\in\mathbb{N}$, putting $s_{k}=\sum\limits_{n=1}^{k}\left(y_{n}\cdot \prod\limits_{j=1}^{m} x_{j,n}^{\ast}(x_{j})\right)$, we have $0\leq s_{k}\uparrow$ and from
$$0\leq s_{k}\leq \sum_{n=1}^{k} \left(y_{n}\cdot \prod_{j=1}^{m}\|x_{j,n}^{\ast}\|\cdot\|x_{j}\|\right)\leq \prod_{j=1}^{m}\|x_{j}\|\cdot \sum_{n=1}^{k} y_{n}=\prod_{j=1}^{m}\|x_{j}\|\cdot \bigvee_{n=1}^{k}y_{n}\leq \prod_{j=1}^{m}\|x_{j}\| y,$$
 it follows that $s_{k}\uparrow z$ for some $z\in F^{+}$ because $F$ is $\sigma$-Dedekind complete. By \cite[Theorem 1.14]{alibur} we get that $z$ is the order limit of $(s_k)_{k=1}^\infty$, in symbols  $z = o{\rm -}\displaystyle\lim_{k} s_{k}$. So, the map
$$A\colon E_{1}^{+}\times\cdots\times E_{m}^{+} \longrightarrow F^{+}~,~ A(x_{1},\ldots,x_{m})=o{\rm -}\lim_{k}\sum_{n=1}^{k}\left(y_{n}\cdot \prod_{j=1}^{m} x_{j,n}^{\ast}(x_{j})\right)$$
is well defined. Calling on \cite[Theorem 1.14]{alibur} once again, the map $A$ is additive in each variable, so by the multilinear Kantorovich Theorem (see \cite[Theorem 2.3]{loane}) there exists a positive $m$-linear operator $B\colon E_{1}\times\cdots\times E_{m}\longrightarrow F$ that extends $A$. Let us see that all Aron-Berner extensions of $B$ are not almost Dunford-Pettis. Note that for $x_{j}\in E_{j}^{+}, j=1,\ldots,m$, and $n\in\mathbb{N}$,
$$B(x_{1},\ldots,x_{m})=A(x_{1},\ldots,x_{m})\geq  y_{n}x_{1,n}^{\ast}(x_{1})\cdots x_{m,n}^{\ast}(x_{m})\geq 0.$$
 Given a positive functional $y^{\ast}\in F^{\ast}$ and $\rho\in S_{m}$, putting $t_{\rho(l)}=x_{\rho(l),n}^{\ast}(x_{\rho(l)}),l=1,\ldots,m-1$, we have
\begin{align*}
(y^{\ast}\circ B(x_{1},& \ldots,_{\rho(1)}  x,\bullet,\ldots,x_{m}))(x_{\rho(1)})=y^{\ast}(B(x_{1},\ldots,x_{m}))\geq y^{\ast}(y_{n}x_{1,n}^{\ast}(x_{1})\cdots x_{m,n}^{\ast}(x_{m}))\\
&\geq  x_{1,n}^{\ast}(x_{1})\cdots x_{m,n}^{\ast}(x_{m})y^{\ast}(y_{n})\geq  x_{1,n}^{\ast}(x_{1})\cdots _{\rho(1)}t\cdots x_{m,n}^{\ast}(x_{m})y^{\ast}(y_{n})t_{\rho(1)},
\end{align*}
which gives $y^{\ast}\circ B(x_{1},\ldots,_{\rho(1)}x,\bullet,\ldots,x_{m}) \geq x_{1,n}^{\ast}(x_{1})\cdots _{\rho(1)}t\cdots x_{m,n}^{\ast}(x_{m})y^{\ast}(y_{n})x_{\rho(1),n}^{\ast}.$ So, for any positive functional $x_{\rho(1)}^{\ast\ast}\in E_{\rho(1)}^{\ast\ast}$,
\begin{align*}
\nonumber\overline{x_{\rho(1)}^{\ast\ast}}^{\rho}(y^{\ast}\circ B)(x_{1},\ldots,&_{\rho(1)}x,_{\rho(2)}x,\bullet,\ldots,x_{m})(x_{\rho(2)})=\overline{x_{\rho(1)}^{\ast\ast}}^{\rho}(y^{\ast}\circ B)(x_{1},\ldots,_{\rho(1)}x,\ldots,x_{m})\\
\nonumber&=x_{\rho(1)}^{\ast\ast}(y^{\ast}\circ B(x_{1},\ldots,_{\rho(1)}x,\bullet,\ldots,x_{m}))\\
&\geq x_{\rho(1)}^{\ast\ast}\big( x_{1,n}^{\ast}(x_{1})\cdots _{\rho(1)}t\cdots x_{m,n}^{\ast}(x_{m})y^{\ast}(y_{n})x_{\rho(1),n}^{\ast}\big)\\
\nonumber&= x_{1,n}^{\ast}(x_{1})\cdots _{\rho(1)}t\cdots x_{m,n}^{\ast}(x_{m})y^{\ast}(y_{n})x_{\rho(1)}^{\ast\ast}(x_{\rho(1),n}^{\ast})\\
\nonumber&= x_{1,n}^{\ast}(x_{1})\cdots _{\rho(1)}t_{\rho(2)}t\cdots x_{m,n}^{\ast}(x_{m})y^{\ast}(y_{n})x_{\rho(1)}^{\ast\ast}(x_{\rho(1),n}^{\ast})t_{\rho(2)}\\
\nonumber&= x_{1,n}^{\ast}(x_{1})\cdots _{\rho(1)}t_{\rho(2)}t\cdots x_{m,n}^{\ast}(x_{m})y^{\ast}(y_{n})x_{\rho(1)}^{\ast\ast}(x_{\rho(1),n}^{\ast})x_{\rho(2),n}^{\ast}(x_{\rho(2)}).
\end{align*}
This implies that, for positive $x_{\rho(1)}^{\ast\ast}\in E_{\rho(1)}^{\ast\ast}$ and $x_{\rho(2)}^{\ast\ast}\in E_{\rho(2)}^{\ast\ast}$, 
\begin{align*}
\nonumber\big(\overline{x_{\rho(2)}^{\ast\ast}}^{\rho}\circ&\overline{x_{\rho(1)}^{\ast\ast}}^{\rho}\big)(y^{\ast}\circ B)(x_{1},\ldots,_{\rho(1)}x,_{\rho(2)}x,_{\rho(3)}x,\bullet,\ldots,x_{m})(x_{\rho(3)})\\
\nonumber&=\big(\overline{x_{\rho(2)}^{\ast\ast}}^{\rho}\circ\overline{x_{\rho(1)}^{\ast\ast}}^{\rho}\big)(y^{\ast}\circ B)(x_{1},\ldots,_{\rho(1)}x,_{\rho(2)}x,\ldots,x_{m})\\
\nonumber&=x_{\rho(2)}^{\ast\ast}\big(\overline{x_{\rho(1)}^{\ast\ast}}^{\rho}(y^{\ast}\circ B)(x_{1},\ldots,_{\rho(1)}x,_{\rho(2)}x,\bullet,\ldots,x_{m})\big)\\
&\geq x_{\rho(2)}^{\ast\ast}\big(x_{1,n}^{\ast}(x_{1})\cdots _{\rho(1)}t_{\rho(2)}t\cdots x_{m,n}^{\ast}(x_{m})y^{\ast}(y_{n})x_{\rho(1)}^{\ast\ast}(x_{\rho(1),n}^{\ast})x^{\ast}_{\rho(2),n}\big)\\
\nonumber&=x_{1,n}^{\ast}(x_{1})\cdots _{\rho(1)}t_{\rho(2)}t\cdots x_{m,n}^{\ast}(x_{m})y^{\ast}(y_{n})x_{\rho(1)}^{\ast\ast}(x_{\rho(1),n}^{\ast})x_{\rho(2)}^{\ast\ast}(x^{\ast}_{\rho(2),n})\\
\nonumber&=x_{1,n}^{\ast}(x_{1})\cdots _{\rho(1)}t_{\rho(2)}t_{\rho(3)}t\cdots x_{m,n}^{\ast}(x_{m})y^{\ast}(y_{n})x_{\rho(1)}^{\ast\ast}(x_{\rho(1),n}^{\ast})x_{\rho(2)}^{\ast\ast}(x^{\ast}_{\rho(2),n})t_{\rho(3)}\\
\nonumber&=x_{1,n}^{\ast}(x_{1})\cdots _{\rho(1)}t_{\rho(2)}t_{\rho(3)}t\cdots x_{m,n}^{\ast}(x_{m})y^{\ast}(y_{n})x_{\rho(1)}^{\ast\ast}(x_{\rho(1),n}^{\ast})x_{\rho(2)}^{\ast\ast}(x^{\ast}_{\rho(2),n})x_{\rho(3),n}^{\ast}(x_{\rho(3)}).
\end{align*}
In the same fashion, for positive $x_{\rho(i)}^{\ast\ast}\in E_{\rho(i)}^{\ast\ast}, i=1,2,3$,
\begin{align*}
\big(\overline{x_{\rho(3)}^{\ast\ast}}&^{\rho}\circ\overline{x_{\rho(2)}^{\ast\ast}}^{\rho}\circ\overline{x_{\rho(1)}^{\ast\ast}}^{\rho}\big)(y^{\ast}\circ B)(x_{1},\ldots,_{\rho(1)}x,_{\rho(2)}x,_{\rho(3)}x,\ldots,x_{m})\\
&=x_{\rho(3)}^{\ast\ast}\big(\big(\overline{x_{\rho(2)}^{\ast\ast}}^{\rho}\circ\overline{x_{\rho(1)}^{\ast\ast}}^{\rho}(y^{\ast}\circ B)\big)(x_{1},\ldots,_{\rho(1)}x,_{\rho(2)}x,_{\rho(3)}x,\bullet,\ldots,x_{m})\big)\\
&\geq x_{\rho(3)}^{\ast\ast}\big(x_{1,n}^{\ast}(x_{1})\cdots _{\rho(1)}t_{\rho(2)}t_{\rho(3)}t\cdots x_{m,n}^{\ast}(x_{m})y^{\ast}(y_{n})x_{\rho(1)}^{\ast\ast}(x_{\rho(1),n}^{\ast})x_{\rho(2)}^{\ast\ast}(x^{\ast}_{\rho(2),n})x_{\rho(3),n}^{\ast}\big)\\
&=x_{1,n}^{\ast}(x_{1})\cdots _{\rho(1)}t_{\rho(2)}t_{\rho(3)}t\cdots x_{m,n}^{\ast}(x_{m})y^{\ast}(y_{n})x_{\rho(1)}^{\ast\ast}(x_{\rho(1),n}^{\ast})x_{\rho(2)}^{\ast\ast}(x^{\ast}_{\rho(2),n}) 
x_{\rho(3)}^{\ast\ast}(x^*_{\rho(3),n}).
\end{align*}
Repeating the process $m-3$ times we obtain that, for all positive $x_{j}^{\ast\ast}\in E_{j}^{\ast\ast}, j=1,\ldots,m$ and every $n\in\mathbb{N},$
\begin{align*}
AB_{m}^{\rho}(B)&(x_{1}^{\ast\ast},\ldots,x_{m}^{\ast\ast})(y^{\ast})=\big(\overline{x_{\rho(m)}^{\ast\ast}}^{\rho}\circ\cdots\circ \overline{x_{\rho(1)}^{\ast\ast}}^{\rho}\big)(y^{\ast}\circ B)\\
&\geq x_{\rho(1)}^{\ast\ast}(x_{\rho(1),n}^{\ast})\cdots x_{\rho(m)}^{\ast\ast}(x^{\ast}_{\rho(m),n})y^{\ast}(y_{n})= x_{1}^{\ast\ast}(x_{1,n}^{\ast})\cdots x_{m}^{\ast\ast}(x^{\ast}_{m,n})J_{F}(y_{n})(y^{\ast}).
\end{align*}
 In particular, for every  $k\in\mathbb{N}$,
 \begin{align*} \|AB_{m}^{\rho}(B)(x_{1,k}^{\ast\ast},\ldots,x_{m,k}^{\ast\ast})\| &\geq\| x_{1,k}^{\ast\ast}(x_{1,k}^{\ast})\cdots x_{m,k}^{\ast\ast}(x_{m,k}^{\ast})J_{E}(y_{k}\|) \geq \frac{1}{2^{m}}\| J_{E}(y_{k})\| =   \frac{1}{2^{m}} > 0,
 \end{align*}
 proving that $AB_{m}^{\rho}(B)$ fails to be almost Dunford-Pettis.
\end{proof}

Now we are in the position to give examples of almost Dunford-Pettis multilinear operators whose Aron-Berner extensions are not almost Dunford-Pettis.

\begin{example}\rm \label{exem5}
 As mentioned before, the Banach lattice  $E =\big(\oplus_{n}\ell_{\infty}^{n}\big)_{1}$ has the positive Schur property and its bidual fails this property. Let $m \geq 2$ and $E_2, \ldots, E_m$ be Banach lattices whose biduals fail the positive Schur property, for example, $E_j = E$ or $E_j = C(K_j)$ where $K_j$ is compact Hausdorff. Finally, let $F$ be any dual Banach lattice containing a copy of $\ell_\infty$. Combining Theorem \ref{propo3} and Lemma \ref{le2}, there is a positive almost Dunford-Pettis $m$-linear operator from $E \times E_2 \times \cdots \times E_m$ to $F$ whose Aron-Berner extensions are not almost Dunford-Pettis. 
\end{example}

Next we search for conditions under which almost Dunford-Pettis multilinear operators have almost Dunford-Pettis Aron-Berner extensions.

\begin{definition} \rm An $m$-linear operator $A\colon E_{1}^{\ast}\times\cdots\times E_{m}^{\ast}\longrightarrow G$ is {\it weakly positively limited} if, for every $j\in\{1,\ldots,m\}$, the sequence $(A(x_{1,n}^{\ast},\ldots,x_{j,n}^{\ast},\ldots,x_{m,n}^{\ast}))_{n=1}^\infty$ is norm null in $G$ whenever $(x_{j,n}^{\ast})_{n=1}^{\infty}$ is positive weak$^*$-null in $E_{j}^{\ast}$ and $(x_{i,n}^{\ast})_{n=1}^{\infty}$ is positive bounded in $E_{i}^{\ast}$ for $i \neq j$. 
\end{definition}

\begin{examples}\label{exem3}\rm (a) It is clear that weakly positively limited multilinear operators are almost Dunford-Pettis. Let $A$ be the bilinear operator given by the duality $\ell_1^* = \ell_\infty$, that is, $A\colon c_{0}^{\ast}\times \ell_{1}^{\ast}\longrightarrow\mathbb{R}$, $A(x,y)=\sum\limits_{n=1}^{\infty}x_{n}y_{n}$. On the one hand, $A$ is almost Dunford-Pettis by Lemma \ref{le2} because $c_{0}^{\ast}=\ell_{1}$ has  the Schur property. On the other hand, putting $y_n = e_1 + \cdots + e_n$ we have $(y_{n})_{n=1}^{\infty}$ positive bounded in $\ell_{1}^{\ast}=\ell_{\infty}$, $(e_{n})_{n=1}^{\infty}$ positive weak$^*$ null in $c_{0}^{\ast}$, but $A(e_{n},y_{n})=1$ for every $n$. So, $A$ is not weakly positively limited.\\
(b) Aron-Berner extensions of positive almost Dunford-Pettis operators are not always weakly positively limited. The positive bilinear operator $A\colon \ell_{1}\times c_{0}\longrightarrow \mathbb{R}$ given by the duality $c_0^* = \ell_1$ is clearly almost Dunford-Pettis and, working with the canonical unit vectors, it is not difficult to see that the (unique) Aron-Berner extension of $A$ is not weakly positively limited. 
\end{examples}

Consider the permutation $\theta \in S_m$ given by $\theta(m) = 1$, $\theta(m-1) = 2, \ldots, \theta(1) = m$. For an $m$-linear operator $A \colon E_1 \times \cdots \times E_m \longrightarrow G$, the Aron-Berner extension $AB_m^{\theta}(A)$ of $A$ associated to $\theta$ is usually referred to as $A^{*[m+1]}$. The reason is that in the bilinear case $m=2$ this extension is usually denoted by $A^{***}$, that is, $AB(A)_2^{\theta} = A^{***}$. The identity permutation in $S_m$ is denoted by id.

\begin{theorem} \label{teo1}
Let $E_{1},E_{2}$ be Banach lattices whose duals have order continuous norms. Then the Aron-Berner extensions of every positive almost Dunford-Pettis bilinear form on $E_{1} \times E_{2}$ are weakly positively limited, hence almost Dunford-Pettis.
\end{theorem}

\begin{proof} Consider the positive linear operator $T\colon E_{1}\longrightarrow E_{2}^{\ast}$ given by $T(x)(y)=A(x,y)$. Let $(x_{n})_{n=1}^{\infty}$ be a positive disjoint weakly null sequence in $E_{1}$. 
It is clear that $T(x_{n})\xrightarrow{\,\, \omega\,\, }0$. Every positive disjoint bounded sequence $(y_{n})_{n=1}^{\infty}$ in $E_{2}$ is weakly null by \cite[Theorem 2.4.14]{nieberg} because the norm of $E_{2}^{\ast}$ is order continuous. Since $A$ is almost  Dunford-Pettis, $T(x_{n})(y_{n})=A(x_{n},y_{n})\longrightarrow 0,$ hence $\|T(x_{n})\|\longrightarrow 0$ by \cite[Corollary 2.7]{dofre}. This proves that  $T$ is almost Dunford-Pettis. For all $x^{\ast\ast}\in E_{1}^{\ast\ast}, y^{\ast\ast}\in E_{2}^{\ast\ast}$ and $x\in E_{1}$,
$$T^{\ast}(y^{\ast\ast})(x)=y^{\ast\ast}(T(x))=y^{\ast\ast}(A(x,\bullet))=\overline{y^{\ast\ast}}^{\theta}(A)(x),$$
showing that $T^{\ast}(y^{\ast\ast})=\overline{y^{\ast\ast}}^{\theta}(A)$, therefore $$A^{\ast\ast\ast}(x^{\ast\ast},y^{\ast\ast})=x^{\ast\ast}\big(\overline{y^{\ast\ast}}^{\theta}(A)\big)=x^{\ast\ast}(T^{\ast}(y^{\ast\ast})).$$

   Let $(x_{n}^{\ast\ast})_{n=1}^{\infty}$ be a positive weak$^*$ null sequence in  $E_{1}^{\ast\ast}$ 
   and $(y_{n}^{\ast\ast})_{n=1}^{\infty}$ be a positive bounded sequence in  $E_{2}^{\ast\ast}$, say $\|y_{n}^{\ast\ast}\|\leq M_{1}$ for every $n$. Since $T$ is almost Dunford-Pettis, $\|T^{\ast\ast}(x_{n}^{\ast\ast})\|\longrightarrow 0$  by \cite[Theorem 3.9]{lg2}. So,
$$0\leq A^{\ast\ast\ast}(x_{n}^{\ast\ast},y_{n}^{\ast\ast})=x_{n}^{\ast\ast}(T^{\ast}(y_{n}^{\ast\ast}))=T^{\ast\ast}(x_{n}^{\ast\ast})(y_{n}^{\ast\ast})|\leq M_{1}\|T^{\ast\ast}(x_{n}^{\ast\ast})\|\longrightarrow 0,$$
from which it follows that $A^{\ast\ast\ast}(x_{n}^{\ast\ast},y_{n}^{\ast\ast})\longrightarrow 0$.

Similarly, $A^{\ast\ast\ast}(z_{n}^{\ast\ast},w_{n}^{\ast\ast})\longrightarrow 0$ for every positive weak$^*$ null sequence $(w_{n}^{\ast\ast})_{n=1}^{\infty}$ in  $E_{2}^{\ast\ast}$ and every positive bounded sequence $(z_{n}^{\ast\ast})_{n=1}^{\infty}$ in  $E_{1}^{\ast\ast}$. This proves that $A^{\ast\ast\ast}$ is weakly positively limited. Working with the positive operator
$S\colon E_{2}\longrightarrow E_{1}^{\ast}$ given by $ S(y)(x)=A(x,y)$, analogously one proves that the other Aron-Berner extension of $A$, namely $AB_{2}^{\rm id}(A)$, is weakly positively limited.
\end{proof}

Next we handle the case of vector-valued bilinear operators. Recall that a Banach lattice $E$ has the {\it dual positive Schur property} if every positive weak$^*$-null sequence in $E^*$ is norm null (see, e.g., \cite{wnuk2013}).

\begin{theorem} Suppose that $E_{1}^{\ast}$ and $ E_{2}^{\ast}$ have order continuous norms and that $A\colon E_{1} \times E_{2} \longrightarrow F$ is a positive almost Dunford-Pettis bilinear operator. If either 
every positive disjoint bounded sequence in $F^{\ast}$ is order bounded, or $F^{\ast}$ has the dual positive Schur property, then the Aron Berner extensions of $A$ are weakly positively limited, hence almost Dunford-Pettis.
\end{theorem}

\begin{proof} Let $\rho\in S_{2}$ be given. For any positive functional $\varphi\in F^{\ast}$, $\varphi\circ A $ is a positive almost Dunford-Pettis bilinear form on $E_1 \times E_2$, so
$$AB_{2}^{\rho}(\varphi\circ A)=\varphi^{\ast\ast}\circ AB_{2}^{\rho}(A) \mbox{ (see \cite[Remark 2.3]{lg}})$$ 
is weakly positively limited by Theorem \ref{teo1}. 
Therefore, for every positive weak$^*$-null sequence $(x_{n}^{\ast\ast})_{n=1}^{\infty}$ in $E_{1}^{\ast\ast}$, 
every positive bounded sequence $(y_{n}^{\ast\ast})_{n=1}^{\infty}$ $E_{2}^{\ast\ast}$ and any $\psi\in F^{\ast}$, it holds
\begin{align*}
|AB_{2}^{\rho}(A)(x_{n}^{\ast\ast},y_{n}^{\ast\ast})(\psi)|&\leq AB_{2}^{\rho}(A)(x_{n}^{\ast\ast},y_{n}^{\ast\ast})(|\psi|)
=(|\psi|^{\ast\ast}\circ AB_{2}^{\rho}(A))(x_{n}^{\ast\ast},y_{n}^{\ast\ast})\longrightarrow 0.
\end{align*}
 It follows that $AB_{2}^{\rho}(A)(x_{n}^{\ast\ast},y_{n}^{\ast\ast})\xrightarrow{\,\omega^{\ast}} 0$ em $F^{\ast\ast}$. Analogously,  $AB_{2}^{\rho}(A)(z_{n}^{\ast\ast},w_{n}^{\ast\ast})\xrightarrow{\,\omega^{\ast}} 0$ whenever $(w_{n}^{\ast\ast})_{n=1}^{\infty}$ is positive weak$^*$-null in $E_{2}^{\ast\ast}$ 
 and $(z_{n}^{\ast\ast})_{n=1}^{\infty}$ is positive bounded in  $E_{1}^{\ast\ast}$.

 If $F^{\ast}$ has the dual positive Schur property, then 
 $\|AB_{2}^{\rho}(A)(x_{n}^{\ast\ast},y_{n}^{\ast\ast})\|\longrightarrow 0$, proving that $AB_{2}^{\rho}(A)$ is weakly positively limited.

 Assume now that positive disjoint bounded sequences in $F^{\ast}$ are order bounded. Given a  positive disjoint bounded sequence  $(\varphi_{n})_{n=1}^{\infty}$ in $F^{\ast}$, let $\varphi\in F^{\ast}$ be such that $0\leq \varphi_{n}\leq \varphi$ for every $n\in\mathbb{N}$. Since each functional $AB_{2}^{\rho}(A)(x_{n}^{\ast\ast},y_{n}^{\ast\ast})\colon F^{\ast}\longrightarrow\mathbb{R}$ is positive,
 $$0\leq AB_{2}^{\rho}(A)(x_{n}^{\ast\ast},y_{n}^{\ast\ast})(\varphi_{n})\leq AB_{2}^{\rho}(A)(x_{n}^{\ast\ast},y_{n}^{\ast\ast})(\varphi)\longrightarrow 0,$$
 which allows us to conclude that $ AB_{2}^{\rho}(A)(x_{n}^{\ast\ast},y_{n}^{\ast\ast})(\varphi_{n})\longrightarrow 0$. By \cite[Corollary 2.7]{dofre} it follows that $\|AB_{2}^{\rho}(A)(x_{n}^{\ast\ast},y_{n}^{\ast\ast})\|\longrightarrow 0$. The remaining convergence follows analogously, so 
 $AB_{2}^{\rho}(A)$ is weakly positively limited.
\end{proof}

Note that for the bilinear form $A$ of  Example \ref{exem3}(b), it is not true that  
$A^{\ast\ast\ast}(x_{n}^{\ast\ast},y_{n}^{\ast\ast})\longrightarrow 0$ whenever  $(x_{n}^{\ast\ast})_{n=1}^{\infty}$ is bounded in $\ell_{1}^{\ast\ast}$ and  $(y_{n}^{\ast\ast})_{n=1}^{\infty}$ is positive weak$^*$-null in $c_{0}^{\ast\ast}$. The next result shows that this happens because one of the duals, namely $\ell_{1}^{\ast}=\ell_{\infty}$, fails the positive Schur property. 

\begin{theorem}\label{teo3} Suppose that $E_{1}^{*},\ldots,E_{m}^{*}$ have the positive Schur property and that 
$A$ is a regular $m$-linear form on $E_{1}\times \cdots \times E_{m}$. Then  $A^{\ast[m+1]}(x_{1,n}^{\ast\ast},\ldots,x_{m,n}^{\ast\ast})\longrightarrow 0$ for every positive weak$^*$-null sequence $(x_{m,n}^{\ast\ast})_{n=1}^{\infty}$ in $E_{m}^{\ast\ast}$ 
and all bounded sequences $(x_{k,n}^{\ast\ast})_{n=1}^{\infty}$ in $E_{k}^{\ast\ast}, k=1,\ldots,m-1$. In particular, $A$ and its Aron-Berner extension $A^{\ast[m+1]}$ are almost Dunford-Pettis.
\end{theorem}

\begin{proof} We shall prove the case $m = 4$. The argument will make clear  that the general case is analogous. Recall that for $x_{i}^{\ast\ast}\in E_{i}^{\ast\ast}, i=1,2,3,4$, $A^{\ast[5]}(x_{1}^{\ast\ast},x_{2}^{\ast\ast},x_{3}^{\ast\ast},x_{4}^{\ast\ast})=\big(\overline{x_{1}^{\ast\ast}}^{\theta}\circ \overline{x_{2}^{\ast\ast}}^{\theta}\circ\overline{x_{3}^{\ast\ast}}^{\theta}\circ \overline{x_{4}^{\ast\ast}}^{\theta}\big)(A)$, where  $\overline{x_{i}^{\ast\ast}}^{\theta}\colon \mathcal{L}_{r}(E_{1},\ldots,E_{i})\longrightarrow \mathcal{L}_{r}(E_{1},\ldots,E_{i-1})$ is given by $\overline{x_{i}^{\ast\ast}}^{\theta}(B)(x_{1},\ldots,x_{i-1})=x_{i}^{\ast\ast}(B(x_{1},\ldots,x_{i-1},\bullet)).$ Suppose that $A$ is positive.

{\bf Claim 1.} $\|\overline{x_{4,n}^{\ast\ast}}^{\theta}(A)(x_{1},x_{2},\bullet)\|\longrightarrow 0$ for all $x_{1}\in E_{1}, x_{2}\in E_{2}$ and every positive weak$^*$-null sequence $(x_{4,n}^{\ast\ast})_{n=1}^{\infty}$ in $E_{4}^{\ast\ast}$. 

To see this, note that by \cite[Proposition 2.1]{lg},
\begin{align*}
|\overline{x_{4,n}^{\ast\ast}}^{\theta}(A)(x_{1},x_{2},\bullet)|(x_{3})&\leq \overline{x_{4,n}^{\ast\ast}}^{\theta}(A)(|x_{1}|,|x_{2}|,\bullet)|(x_{3})=x_{4,n}^{\ast\ast}(A(|x_{1}|,|x_{2}|,x_{3},\bullet))\longrightarrow 0,
\end{align*}
hence $|\overline{x_{4,n}^{\ast\ast}}^{\theta}(A)(x_{1},x_{2},\bullet)|\xrightarrow{\,\omega^{\ast}} 0$. The positive linear operator $$T_{3}\colon E_{3}\longrightarrow \mathcal{L}_{r}(E_{1},E_{2};E_{4}^{\ast})~,~T_{3}(x_{3})(x_{1},x_{2})(x_{4})= A(x_{1},x_{2},x_{3},x_{4}),$$
is almost Dunford-Pettis because $\mathcal{L}_{r}(E_{1},E_{2};E_{4}^{\ast})$ has the positive Schur property by \cite[Theorem 2.8]{gebu}. 
 Let $(x_{3,n})_{n=1}^{\infty}$ be a positive disjoint bounded sequence in $E_{3}$. Since $E_{3}^{\ast}$ has the positive Schur property, its norm is order continuous, so $x_{3,n}\xrightarrow{\,\omega\,}0$ by \cite[Theorem 2.4.14]{nieberg}. Therefore, assim $\|T_{3}(x_{3,n})\|\longrightarrow 0$ and
\begin{align*}
|\overline{x_{4,n}^{\ast\ast}}^{\theta}(A)(x_{1},x_{2},\bullet)(x_{3,n})|&\leq \overline{x_{4,n}^{\ast\ast}}^{\theta}(A)(|x_{1}|,|x_{2}|,x_{3,n})|=x_{4,n}^{\ast\ast}(A(|x_{1}|,|x_{2}|,x_{3,n},\bullet))\\
&= x_{4,n}^{\ast\ast}(T_{3}(x_{3,n})(|x_{1}|,|x_{2}|)\leq \|x_{4,n}^{\ast\ast}\|\cdot \|x_{1}\|\cdot \|x_{2}\|\cdot \|T_{3}(x_{3,n})\|\longrightarrow 0.
\end{align*}
Now Claim 1 follows from \cite[Corollary 2.7]{dofre}. 

{\bf Claim 2.} $\|\big(\overline{x_{3,n}^{\ast\ast}}^{\theta}\circ \overline{x_{4,n}^{\ast\ast}}^{\theta}\big)(A)(x_{1},\bullet)\|\longrightarrow 0$ for every $x_{1}\in E_{1}$, every bounded sequence  $(x_{3,n}^{\ast\ast})_{n=1}^{\infty}$ in $E_{3}^{\ast\ast}$ and every positive weak$^*$-null sequence $(x_{4,n}^{\ast\ast})_{n=1}^{\infty}$ in $E_{4}^{\ast\ast}$. 

Indeed, calling on \cite[Proposition 2.1]{lg} and Claim 1,
\begin{align*}
|\big(\overline{x_{3,n}^{\ast\ast}}^{\theta}\circ \overline{x_{4,n}^{\ast\ast}}^{\theta}\big)(A)(x_{1},&\bullet)|(x_{2})\leq\big(\overline{|x_{3,n}^{\ast\ast}|}^{\theta}\circ \overline{x_{4,n}^{\ast\ast}}^{\theta}\big)(A)(|x_{1}|,x_{2})=\overline{|x_{3,n}^{\ast\ast}|}^{\theta}\big( \overline{x_{4,n}^{\ast\ast}}^{\theta}(A)\big)(|x_{1}|,x_{2})\\
&=|x_{3,n}^{\ast\ast}|\big(\overline{x_{4,n}^{\ast\ast}}^{\theta}(A)(|x_{1}|,x_{2},\bullet)\big)\leq \|x_{3,n}^{\ast\ast}\|\cdot\|\overline{x_{4,n}^{\ast\ast}}^{\theta}(A)(|x_{1}|,x_{2},\bullet)\|\longrightarrow 0,
\end{align*}
hence $|\big(\overline{x_{3,n}^{\ast\ast}}^{\theta}\circ \overline{x_{4,n}^{\ast\ast}}^{\theta}\big)(A)(x_{1},\bullet)|\xrightarrow{\,\omega^{\ast}} 0$ in $E_{2}^{\ast}$. Let $T_{2} \colon E_{2} \longrightarrow \mathcal{L}_{r}(E_{1},E_{3};E_{4}^{\ast})$ be the positive linear operator such that 
$A(x_{1},x_{2},x_{3},x_{4})=T_{2}(x_{2})(x_{1},x_{3})(x_{4})$. Thus $T_{2}$ is almost Dunford-Pettis because $\mathcal{L}_{r}(E_{1},E_{3};E_{4}^{\ast})$ has the positive Schur property by \cite[Theorem 2.8]{gebu}. Let $(x_{2,n})_{n=1}^{\infty}$ be a positive disjoint bounded sequence in $E_{2}$. As before, $E_{2}^{\ast}$ has order continuous norm, so $x_{2,n}\xrightarrow{\,\omega\,}0$ by \cite[Theorem 2.4.14]{nieberg}. Therefore, $\|T_{2}(x_{2,n})\|\longrightarrow 0$ and
\begin{align*}
|\overline{x_{4,n}^{\ast\ast}}^{\theta}(A)(|x_{1}|,x_{2,n},\bullet)&(x_{3})|\leq \overline{x_{4,n}^{\ast\ast}}^{\theta}(A)(|x_{1}|,x_{2,n},|x_{3}|)|= x_{4,n}^{\ast\ast}(A(|x_{1}|,x_{2,n},|x_{3}|,\bullet))\\
&= x_{4,n}^{\ast\ast}(T_{2}(x_{2,n})(|x_{1}|,|x_{3}|))\leq \|x_{4,n}^{\ast\ast}\|\cdot\|x_{1}\|\cdot\|x_{3}\|\cdot\|T_{2}(x_{2,n}\|\longrightarrow 0,
\end{align*}
hence $\overline{x_{4,n}^{\ast\ast}}^{\theta}(A)(|x_{1}|,x_{2,n},\bullet)\xrightarrow{\,\omega^{\ast}}0$. For any positive disjoint bounded sequence $(x_{3,n})_{n=1}^{\infty}$ in $E_{3}$,
\begin{align*}
\overline{x_{4,n}^{\ast\ast}}^{\theta}(A)(|x_{1}|,x_{2,n},&\bullet)(x_{3,n})=\overline{x_{4,n}^{\ast\ast}}^{\theta}(A)(|x_{1}|,x_{2,n},x_{3,n})\leq x_{4,n}^{\ast\ast}(A(|x_{1}|,x_{2,n},x_{3,n},\bullet))\\
&=x_{4,n}^{\ast\ast}(T_{2}(x_{2,n})(x_{1},x_{3,n}))\leq \|x_{4,n}^{\ast\ast}\| \cdot \|x_{1}\|\cdot \|x_{3,n}\|\cdot \|T_{2}(x_{2,n})\|\longrightarrow 0.
\end{align*}
By \cite[Corollary 2.7]{dofre} we have $\|\overline{x_{4,n}^{\ast\ast}}^{\theta}(A)(|x_{1}|,x_{2,n},\bullet)\|\longrightarrow 0$. It follows that
\begin{align*}
|\big(\overline{x_{3,n}^{\ast\ast}}^{\theta}\circ \overline{x_{4,n}^{\ast\ast}}^{\theta}\big)(A)&(x_{1},\bullet)(x_{2,n})|=|\big(\overline{x_{3,n}^{\ast\ast}}^{\theta}\circ \overline{x_{4,n}^{\ast\ast}}^{\theta}\big)(A)(x_{1},x_{2,n})|\leq\overline{|x_{3,n}^{\ast\ast}|}^{\theta}\big( \overline{x_{4,n}^{\ast\ast}}^{\theta}(A)\big)(|x_{1}|,x_{2,n})\\
&=|x_{3,n}^{\ast\ast}|\big(\overline{x_{4,n}^{\ast\ast}}^{\theta}(A)(|x_{1}|,x_{2,n},\bullet)\big)\leq\|x_{3,n}^{\ast\ast}\|\cdot\|\overline{x_{4,n}^{\ast\ast}}^{\theta}(A)(|x_{1}|,x_{2,n},\bullet)\|\longrightarrow 0.
\end{align*}
Calling on \cite[Corollary 2.7]{dofre} once again we get Claim 2. 

{\bf Claim 3.} $\|\big(\overline{x_{2,n}^{\ast\ast}}^{\theta}\circ \overline{x_{3,n}^{\ast\ast}}^{\theta}\circ \overline{x_{4,n}^{\ast\ast}}^{\theta}\big)(A)\|\longrightarrow 0$ for all bounded sequences $(x_{j,n}^{\ast\ast})_{n=1}^{\infty}$ in $E_{j}^{\ast\ast}, j=2,3$, and every positive weak$^*$-null sequence $(x_{4,n}^{\ast\ast})_{n=1}^{\infty}$ in $E_{4}^{\ast\ast}$.

By Claim 2,
\begin{align*}
|\big(\overline{x_{2,n}^{\ast\ast}}^{\theta}&\circ\overline{x_{3,n}^{\ast\ast}}^{\theta}\circ \overline{x_{4,n}^{\ast\ast}}^{\theta}\big)(A)(x_{1})|=|\overline{x_{2,n}^{\ast\ast}}^{\theta}\big(\big( \overline{x_{3,n}^{\ast\ast}}^{\theta}\circ \overline{x_{4,n}^{\ast\ast}}^{\theta}\big)(A)\big)(x_{1})|\\
&=|x_{2,n}^{\ast\ast}\big(\big( \overline{x_{3,n}^{\ast\ast}}^{\theta}\circ \overline{x_{4,n}^{\ast\ast}}^{\theta}\big)(A)(x_{1},\bullet)\big)|\leq \|x_{2,n}^{\ast\ast}\|\cdot\|\big( \overline{x_{3,n}^{\ast\ast}}^{\theta}\circ \overline{x_{4,n}^{\ast\ast}}^{\theta}\big)(A)(x_{1},\bullet)\|\longrightarrow 0,
\end{align*}
which gives $\big(\overline{x_{2,n}^{\ast\ast}}^{\theta}\circ\overline{x_{3,n}^{\ast\ast}}^{\theta}\circ \overline{x_{4,n}^{\ast\ast}}^{\theta}\big)(A)\xrightarrow{\omega^{\ast}} 0$. Let us see that, for every positive disjoint bounded sequence $(x_{1,n})_{n=1}^{\infty}$ in $E_{1}$,  $\big(\overline{x_{2,n}^{\ast\ast}}^{\theta}\circ\overline{x_{3,n}^{\ast\ast}}^{\theta}\circ \overline{x_{4,n}^{\ast\ast}}^{\theta}\big)(A)(x_{1,n})\longrightarrow 0$. Let $T_{1}\colon E_{1} \longrightarrow \mathcal{L}_{r}(E_{2},E_{3};E_{4}^{\ast})$ be the positive linear operator such that  $T_{1}(x_{1})(x_{2},x_{3})(x_{4})= A(x_{1},x_{2},x_{3},x_{4})$. Applying 
\cite[Theorem 2.8]{gebu}, $\mathcal{L}_{r}(E_{2},E_{3};E_{4}^{\ast})$ has the positive Schur property, hence $T_{1}$ is almost Duford-Pettis. Since $E_{1}^{\ast}$ has order continuous norm, $x_{1,n}\xrightarrow{\,\omega\,} 0$ by \cite[Theorem 2.4.14]{nieberg}, hence $\|T_{1}(x_{1,n})\|\longrightarrow 0$. By \cite[Proposition 2.1]{lg},
\begin{align*}
|\overline{x_{4,n}^{\ast\ast}}^{\theta}(A)(x_{1,n},x_{2},\bullet)|(x_{3})&\leq \overline{x_{4,n}^{\ast\ast}}^{\theta}(A)(x_{1,n},|x_{2}|,x_{3})=x_{4,n}^{\ast\ast}(A(x_{1,n},|x_{2}|,x_{3},\bullet))\\
&=x_{4,n}^{\ast\ast}(T_{1}(x_{1,n})(|x_{2}|,x_{3})) \leq \|x_{4,n}^{\ast\ast}\|\cdot\|x_{2}\|\cdot\|x_{3}\|\cdot\|T_{1}(x_{1,n})\|\longrightarrow 0,
\end{align*}
from which we get $|\overline{x_{4,n}^{\ast\ast}}^{\theta}(A)(x_{1,n},x_{2},\bullet)|\xrightarrow{\,\omega^{\ast}} 0$. For any positive disjoint bounded sequence $(x_{3,n})_{n=1}^{\infty}$ in $E_{3}$,
\begin{align*}
|\overline{x_{4,n}^{\ast\ast}}^{\theta}(A)(x_{1,n},& x_{2},\bullet)(x_{3,n})|=|\overline{x_{4,n}^{\ast\ast}}^{\theta}(A)(x_{1,n},x_{2},x_{3,n})|=|x_{4,n}^{\ast\ast}(A(x_{1,n},x_{2},x_{3,n},\bullet))|\\
&=|x_{4,n}^{\ast\ast}(T_{1}(x_{1,n})(x_{2},x_{3,n}))|\leq \|x_{4,n}^{\ast\ast}\|\cdot\|x_{2}\|\cdot\|x_{3,n}\|\cdot\|T_{1}(x_{1,n}\|\longrightarrow 0.
\end{align*}
By \cite[Corollary 2.7]{dofre}, $\|\overline{x_{4,n}^{\ast\ast}}^{\theta}(A)(x_{1,n},x_{2},\bullet)\|\longrightarrow 0$, and by \cite[Proposition 2.1]{lg},
\begin{align*}
|\big(\overline{x_{3,n}^{\ast\ast}}^{\theta}\circ \overline{x_{4,n}^{\ast\ast}}^{\theta}\big)&(A)(x_{1,n},\bullet)|(x_{2})\leq\big(\overline{|x_{3,n}^{\ast\ast}|}^{\theta}\circ \overline{x_{4,n}^{\ast\ast}}^{\theta}\big)(A)(x_{1,n},x_{2})=\overline{|x_{3,n}^{\ast\ast}|}^{\theta}\big( \overline{x_{4,n}^{\ast\ast}}^{\theta}(A)\big)(x_{1,n},x_{2})\\
&=|x_{3,n}^{\ast\ast}|\big(\overline{x_{4,n}^{\ast\ast}}^{\theta}(A)(x_{1,n},x_{2},\bullet)\big)\leq \|x_{3,n}^{\ast\ast}\|\cdot\|\overline{x_{4,n}^{\ast\ast}}^{\theta}(A)(x_{1,n},|x_{2}|,\bullet)\|\longrightarrow 0.
\end{align*}
It follows that $|\big(\overline{x_{3,n}^{\ast\ast}}^{\theta}\circ \overline{x_{4,n}^{\ast\ast}}^{\theta}\big)(A)(x_{1,n},\bullet)|\xrightarrow{\,\omega^{\ast}} 0$ in $E_{2}^{\ast}$. For any positive disjoint bounded sequence $(x_{2,n})_{n=1}^{\infty}$ in $E_{2}$,
\begin{align*}
|\overline{x_{4,n}^{\ast\ast}}^{\theta}(A)(x_{1,n},x_{2,n},& \bullet)(x_{3})|=|\overline{x_{4,n}^{\ast\ast}}^{\theta}(A)(x_{1,n},x_{2,n},x_{3})|\leq x_{4,n}^{\ast\ast}(A(x_{1,n},x_{2,n},|x_{3}|,\bullet))\\
&\leq x_{4,n}^{\ast\ast}(T_{1}(x_{1,n})(x_{2,n},|x_{3}|))\leq \|x_{4,n}^{\ast\ast}\|\cdot\|x_{2,n}\|\cdot\|x_{3}\|\cdot \|T_{1}(x_{1,n})\|\longrightarrow 0,
\end{align*}
which gives $\overline{x_{4,n}^{\ast\ast}}^{\theta}(A)(x_{1,n},x_{2,n},\bullet)\xrightarrow{\,\omega^{\ast}}0$. And for any positive disjoint bounded sequence $(x_{3,n})_{n=1}^{\infty}$ in $E_{3}$,
\begin{align*}
0\leq \overline{x_{4,n}^{\ast\ast}}^{\theta}(A)&(x_{1,n},x_{2,n},\bullet)(x_{3,n})=\overline{x_{4,n}^{\ast\ast}}^{\theta}(A)(x_{1,n},x_{2,n},x_{3,n})\leq x_{4,n}^{\ast\ast}(A(x_{1,n},x_{2,n},x_{3,n},\bullet))\\
&=x_{4,n}^{\ast\ast}(T_{1}(x_{1,n})(x_{2,n},x_{3,n}))\leq \|x_{4,n}^{\ast\ast}\| \cdot\|x_{2,n}\| \cdot \|x_{3,n}\| \cdot \|T_{1}(x_{1,n})\|\longrightarrow 0,
\end{align*}
proving that  $\overline{x_{4,n}^{\ast\ast}}^{\theta}(A)(x_{1,n},x_{2,n},\bullet)(x_{3,n})\longrightarrow 0$. We get $\|\overline{x_{4,n}^{\ast\ast}}^{\theta}(A)(x_{1,n},x_{2,n},\bullet)\|\longrightarrow 0$ by \cite[Corollary 2.7]{dofre}, so
\begin{align*}
|\big(\overline{x_{3,n}^{\ast\ast}}^{\theta}\circ \overline{x_{4,n}^{\ast\ast}}^{\theta}\big)(A)(x_{1,n},\bullet)&(x_{2,n})|=|\big(\overline{x_{3,n}^{\ast\ast}}^{\theta}\circ \overline{x_{4,n}^{\ast\ast}}^{\theta}\big)(A)(x_{1,n},x_{2,n})|\\
&\leq\overline{|x_{3,n}^{\ast\ast}|}^{\theta}\big( \overline{x_{4,n}^{\ast\ast}}^{\theta}(A)\big)(x_{1,n},x_{2,n})=|x_{3,n}^{\ast\ast}|\big(\overline{x_{4,n}^{\ast\ast}}^{\theta}(A)(x_{1,n},x_{2,n},\bullet)\big)\\
&\leq \|x_{3,n}^{\ast\ast}\|\cdot\|\overline{x_{4,n}^{\ast\ast}}^{\theta}(A)(x_{1,n},x_{2,n},\bullet)\|\longrightarrow 0.
\end{align*}
\cite[Corollary 2.7]{dofre} implies $\|\big(\overline{x_{3,n}^{\ast\ast}}^{\theta}\circ \overline{x_{4,n}^{\ast\ast}}^{\theta}\big)(A)(x_{1,n},\bullet)\|\longrightarrow 0$, therefore
\begin{align*}
|\big( \overline{x_{2,n}^{\ast\ast}}^{\theta}\circ\overline{x_{3,n}^{\ast\ast}}^{\theta}\circ \overline{x_{4,n}^{\ast\ast}}^{\theta}\big)(A)(x_{1,n})|&=|\overline{x_{2,n}^{\ast\ast}}^{\theta}\big(\big( \overline{x_{3,n}^{\ast\ast}}^{\theta}\circ \overline{x_{4,n}^{\ast\ast}}^{\theta}\big)(A)\big)(x_{1,n})|\\
&=|x_{2,n}^{\ast\ast}\big(\big( \overline{x_{3,n}^{\ast\ast}}^{\theta}\circ \overline{x_{4,n}^{\ast\ast}}^{\theta}\big)(A)(x_{1,n},\bullet)\big)|\\
&\leq\|x_{2,n}^{\ast\ast}\|\cdot\|\big( \overline{x_{3,n}^{\ast\ast}}^{\theta}\circ \overline{x_{4,n}^{\ast\ast}}^{\theta}\big)(A)(x_{1,n},\bullet)\|\longrightarrow 0,
\end{align*}
which gives $\big(\overline{x_{2,n}^{\ast\ast}}^{\theta}\circ\overline{x_{3,n}^{\ast\ast}}^{\theta}\circ \overline{x_{4,n}^{\ast\ast}}^{\theta}\big)(A)(x_{1,n})\longrightarrow 0$. Calling  \cite[Corollary 2.7]{dofre} one last time, we obtain Claim 3. 

  Finally, let $(x_{j,n}^{\ast\ast})_{n=1}^{\infty}$ be bounded sequences in $E_{j}^{\ast\ast}, j=1,2,3$, and $(x_{4,n}^{\ast\ast})_{n=1}^{\infty}$ be a positive weak$^*$-null sequence in $E_{4}^{\ast\ast}$. 
  By Claim 3,
\begin{align*}
|A^{\ast[5]}(x_{1,n}^{\ast\ast},\ldots,x_{4,n}^{\ast\ast})|&=|\big(\overline{x_{1,n}^{\ast\ast}}^{\theta}\circ\cdots\circ \overline{x_{4,n}^{\ast\ast}}^{\theta}\big)(A)|=|x_{1,n}^{\ast\ast}\big(\big(\overline{x_{2,n}^{\ast\ast}}^{\theta}\circ\overline{x_{3,n}^{\ast\ast}}^{\theta}\circ \overline{x_{4,n}^{\ast\ast}}^{\theta}\big)(A)\big)|\\
&\leq \|x_{1,n}^{\ast\ast}\|\cdot\|\big(\overline{x_{2,n}^{\ast\ast}}^{\theta}\circ\overline{x_{3,n}^{\ast\ast}}^{\theta}\circ \overline{x_{4,n}^{\ast\ast}}^{\theta}\big)(A)\|\longrightarrow 0,
\end{align*}
proving that $A^{\ast[5]}(x_{1,n}^{\ast\ast},\ldots,x_{4,n}^{\ast\ast})\longrightarrow 0$.

If $A$ is regular, then $A=A_{1}-A_{2}$ for some positive $m$-linear operators $A_1$ and $A_2$. Apply what we have just proved for $A_1$ and $A_2$ and use that 
$A^{\ast[5]}=A_{1}^{\ast[5]}-A_{2}^{\ast[5]}$. 
\end{proof}

The case of vector-valued multilinear operators reads as follows.

\begin{theorem}\label{teo6} Suppose that $E_{1}^{\ast},\ldots,E_{m}^{\ast}$ have the positive Schur property and that either $F^*$ has the dual positive Schur property or every positive disjoint bounded sequence in $F^*$ is order bounded. If $A \colon E_{1} \times \cdots \times E_{m}\longrightarrow F$ is a regular $m$-linear operator, then 
$\|A^{\ast[m+1]}(x_{1,n}^{\ast\ast},\ldots,x_{m,n}^{\ast\ast})\|\longrightarrow 0$ for every positive weak$^*$-null sequence $(x_{m,n}^{\ast\ast})_{n=1}^{\infty}$ in $E_{m}^{\ast\ast}$ 
and all bounded sequences $(x_{k,n}^{\ast\ast})_{n=1}^{\infty}$ in $E_{k}^{\ast\ast}, k=1,\ldots,m-1$. In particular, $A$ and its Aron-Berner extension $A^{\ast[m+1]}$ are almost Dunford-Pettis.
\end{theorem}

\begin{proof}
Write $A=A_{1}-A_{2}$ where $A_{1}, A_{2}$ are positive $m$-linear operators. Let $(x_{m,n}^{\ast\ast})_{n=1}^{\infty}$ be a positive weak$^*$-null sequence in $E_{m}^{\ast\ast}$ and $(x_{k,n}^{\ast\ast})_{n=1}^{\infty}$ be bounded sequences in $E_{k}^{\ast\ast}, k=1,\ldots,m-1$. For each functional 
$y^{\ast}\in F^{\ast}$, $y^{\ast}\circ A_{i}\in\mathcal{L}_{r}(E_{1},\ldots,E_{m}), i=1,2$, therefore  $(y^{\ast}\circ A_{i})^{\ast[m+1]}(|x_{1,n}^{\ast\ast}|,\ldots,|x_{m-1,n}^{\ast\ast}|, x_{m,n}^{\ast\ast})\longrightarrow 0$ by Theorem \ref{teo3}. From 
\begin{align*}
|\,|A^{\ast[m+1]}(x_{1,n}^{\ast\ast},\ldots,x_{m,n}^{\ast\ast})|(y^{\ast})|&\leq |A^{\ast[m+1]}|(|x_{1,n}^{\ast\ast}|,\ldots, |x_{m-1,n}^{\ast\ast}|, x_{m,n}^{\ast\ast})(|y^{\ast}|)\\
&=|A_{1}^{\ast[m+1]}-A_{2}^{\ast[m+1]}|(|x_{1,n}^{\ast\ast}|,\ldots,|x_{m-1,n}^{\ast\ast}|,x_{m,n}^{\ast\ast})(|y^{\ast}|)\\
&\leq \big(A_{1}^{\ast[m+1]}+A_{2}^{\ast[m+1]}\big)(|x_{1,n}^{\ast\ast}|,\ldots,|x_{m-1,n}^{\ast\ast}|,x_{m,n}^{\ast\ast})(|y^{\ast}|)\\
&=\sum_{i=1}^{2}(|y^{\ast}|\circ A_{i})^{\ast[m+1]}(|x_{1,n}^{\ast\ast}|,\ldots,|x_{m-1,n}^{\ast\ast}|,x_{m,n}^{\ast\ast})\longrightarrow 0,
\end{align*}
it follows that the positive sequence  $(|A^{\ast[m+1]}(x_{1,n}^{\ast\ast},\ldots,x_{m,n}^{\ast\ast})|)_{n=1}^{\infty}$ is weak$^*$-null in $F^{\ast\ast}$. If $F^*$ has the dual positive Schur property, then 
$\|A^{\ast[m+1]}(x_{1,n}^{\ast\ast},\ldots,x_{m,n}^{\ast\ast})\|\longrightarrow 0$ as desired. Suppose now that positive disjoint bounded sequences in $F^*$ are order bounded. Given a positive disjoint bounded sequence $(\varphi_{n})_{n=1}^{\infty}$ in $F^{\ast}$, let $\varphi\in F^{\ast}$ be such that $0\leq \varphi_{n}\leq \varphi$ for every $n$. From
\begin{align*}
|A^{\ast[m+1]}(x_{1,n}^{\ast\ast}&,\ldots,x_{m,n}^{\ast\ast})(\varphi_{n})|\leq |A^{\ast[m+1]}|(|x_{1,n}^{\ast\ast}|,\ldots,|x_{m-1,n}^{\ast\ast}|,x_{m,n}^{\ast\ast})(\varphi_{n})\\
& \leq |A^{\ast[m+1]}|(|x_{1,n}^{\ast\ast}|,\ldots,|x_{m-1,n}^{\ast\ast}|,x_{m,n}^{\ast\ast})(\varphi)\\
&=|A_{1}^{\ast[m+1]}-A_{2}^{\ast[m+1]}|(|x_{1,n}^{\ast\ast}|,\ldots,|x_{m-1,n}^{\ast\ast}|,x_{m,n}^{\ast\ast})(\varphi)\\
&\leq A_{1}^{\ast[m+1]}(|x_{1,n}^{\ast\ast}|,\ldots,|x_{m-1,n}^{\ast\ast}|,x_{m,n}^{\ast\ast})(\varphi)+A_{2}^{\ast[m+1]}(|x_{1,n}^{\ast\ast}|,\ldots,|x_{m-1,n}^{\ast\ast}|,x_{m,n}^{\ast\ast})(\varphi)\\
&=\big((\varphi\circ A_{1})^{\ast[m+1]}+(\varphi\circ A_{2})^{\ast[m+1]}\big)(|x_{1,n}^{\ast\ast}|,\ldots,|x_{m-1,n}^{\ast\ast}|,x_{m,n}^{\ast\ast})\longrightarrow 0,
\end{align*}
we obtain that $A^{\ast[m+1]}(x_{1,n}^{\ast\ast},\ldots,x_{m,n}^{\ast\ast})(\varphi_{n})\longrightarrow 0$. Therefore,  $\|A^{\ast[m+1]}(x_{1,n}^{\ast\ast},\ldots,x_{m,n}^{\ast\ast})\|\longrightarrow 0$ by \cite[Corollary 2.7]{dofre}.

 In particular, in both cases $A^{\ast[m+1]}$ is almost Dunford-Pettis, and Lemma  \ref{le2}(a) assures that $A$ is almost Dunford-Pettis as well.
\end{proof}

\begin{remark}\rm With the obvious adjustments, the conclusions of Theorems \ref{teo3} and \ref{teo6} hold for all Aron-Berner extensions $AB_m^\rho(A), \rho \in S_m$, of $A$.
\end{remark}

\begin{examples}\rm As  to the assumptions of Theorems \ref{teo3} and \ref{teo6}, we mention some illustrative examples: (i) AL-spaces have the positive Schur property \cite[p.\,16]{wnukbari}, so duals of AM-spaces have this property. In particular, duals of $C(K)$-spaces and of $L_1(\mu)^*$-spaces have the positive Schur property. In particular, the dual of $\ell_\infty$ has the positive Schur property. Moreover, if $E_1^*, \ldots, E_m^*$ have the positive Schur property, then the dual of the positive projective tensor product $E_1 \widehat\otimes_{|\pi|} \otimes \cdots \widehat\otimes_{|\pi|} E_m$ has the positive Schur property as well \cite[Proposition 4.3]{gebu}. 
(ii) $C(K)$-spaces have the dual positive Schur property \cite[Ex.\,4]{wnuk2013} and the quotient of a Banach lattice with this property  by a closed ideal has this property as well \cite[p.\,763]{wnuk2013}. In particular, $\ell_\infty/c_0$ has the dual positive Schur property. (iii) It is clear that positive disjoint bounded sequences in $C(K)$-spaces are order bounded.
\end{examples}

\bigskip

\noindent Faculdade de Matem\'atica~~~~~~~~~~~~~~~~~~~~~~Instituto de Matem\'atica e Estat\'istica\\
Universidade Federal de Uberl\^andia~~~~~~~~ Universidade de S\~ao Paulo\\
38.400-902 -- Uberl\^andia -- Brazil~~~~~~~~~~~~ 05.508-090 -- S\~ao Paulo -- Brazil\\
e-mail: botelho@ufu.br ~~~~~~~~~~~~~~~~~~~~~~~~~e-mail: luisgarcia@ime.usp.br

%
%
%

\end{document}